\documentclass{amsart}
\usepackage{url}
\usepackage[version=3]{mhchem}
\usepackage{pstricks}

\newtheorem{theorem}{Theorem}[section]

\newtheorem{corollary}[theorem]{Corollary}

\newtheorem{proposition}[theorem]{Proposition}
\theoremstyle{definition}

\newtheorem{example} [theorem] {Example}
\newtheorem{remark} [theorem] {Remark}

\numberwithin{equation}{section}

\newcommand{\A}{\mathbb{A}}

\newcommand{\K}{\mathbb{K}}
\newcommand{\PP}{\mathbb{P}}
\newcommand{\Q}{\mathbb{Q}}
\newcommand{\R}{\mathbb{R}}
\newcommand{\Z}{\mathbb{Z}}

\begin{document}

\title{Multi-Rees Algebras and Toric Dynamical Systems}

\author{David A. Cox}
\address{Department of Mathematics \& Statistics, Amherst College, Amherst, MA 01002}
\email{dacox@amherst.edu}

\author{Kuei-Nuan Lin}
\address{The Penn State University, Greater Allegheny Campus, McKeesport, PA 15132}
\email{kul20@psu.edu}

\author{Gabriel Sosa}
\address{Department of Mathematics \& Statistics, Amherst College, Amherst, MA 01002}
\email{gsosa@amherst.edu}


\date{\today}

\commby{Claudia Polini}

\begin{abstract}
This paper explores the relation between multi-Rees algebras and ideals that arise in the study of toric dynamical systems from the theory of chemical reaction networks.   
\end{abstract}

\date{\today}
\maketitle

\section*{Introduction}
\label{introduction}

Rees algebras are a well-established subject in commutative algebra, going back to the 1960s, while the interactions between chemical reaction networks and algebraic geometry are more recent, beginning with the pioneering work of Karin Gatermann in the early 2000s (see, for example, \cite{gatermann4}).  The goal of this paper is to explore some links between these two fields.

Given a commutative ring $R$ and a collection of ideals $I_1$, $I_2$, $\dots$, $I_l$ of $R$, the \textit{multi-Rees algebra of} $I_1$, $I_2$, $\dots$, $I_l$ (which is also the Rees algebra of the module $M=I_1{\oplus} \cdots {\oplus} I_l$) is defined as the multigraded $R$-algebra:
\[
\mathcal{R}_R(I_1{\oplus} \cdots {\oplus} I_l) = \bigoplus_{a_1,\dots,a_l \ge 0} I_1^{a_1} \cdots I_l^{a_l} t_1^{a_1}\cdots t_l^{a_l} \subseteq R[t_1,\dots,t_l]
\]
for auxiliary variables $t_1,\dots,t_l$. In the particular case when $l=1$, this gives the classical Rees algebra of an ideal $I$ in a commutative ring $R$, often written 
\[
\mathcal{R}_R(I) = R \oplus I \oplus I^2 \oplus I^3 \oplus \cdots
\]
or $\mathcal{R}_R(I) = R[It] \subseteq R[t]$. Some recent papers on multi-Rees algebras include \cite{BC,J,LinCM,SosaMulti}. 

In the study of chemical reaction networks, the toric methods
introduced by Gatermann were formalized in the 2009 paper \emph{Toric
  Dynamical Systems} \cite{TDS}.  We also recommend the paper \cite{MCA} for more on the algebraic geometry of chemical and biochemical reaction networks.

Our main result is Theorem~\ref{TTG}, which shows that if $G$ is the digraph associated to a chemical reaction network with $l$ strongly connected components then the toric ideal $T_G$, defined in \cite{TDS}, is the defining ideal of the multi-Rees algebra $\mathcal{R}_R(I_1{\oplus} \cdots {\oplus} I_l)$, with $I_i$ the monomial ideal generated by the monomials representing the chemical complexes involved in the $i^{\text{th}}$ strongly connected component of $G$. 

The paper is organized as follows.  In Section~\ref{notation}, we set up our notation for the multi-Rees algebra of a collection of monomial ideals and introduce the \emph{defining ideal} of a multi-Rees algebra.  Section~\ref{results} gives an explicit formula for the defining ideal and proves that it is a toric ideal.  In Section~\ref{application}, we recall the definition of a chemical reaction network and apply the results of the previous section to prove Theorem~\ref{TTG}.  We also explore the effect of adding edges to the network and discuss some further results and definitions from \cite{TDS}.  Then Section~\ref{cayley} relates the Cayley matrix to the toric ideal $T_G$ and the $R$-algebra generators of the multi-Rees algebra.  Finally, in Section~\ref{special}, we recall the \emph{special fiber} of a multi-Rees algebra and study its relation to the moduli ideal $M_G$ defined in \cite{TDS}.

\section{Notation and Definitions}
\label{notation}

 Fix a field $\K$ and consider the polynomial ring $R = \K[\mathbf{x}]$ for variables $\mathbf{x} = (x_1,\dots,x_s)$.  Assume that we have exponent vectors $\mathbf{v}_1,\dots,\mathbf{v}_n \in \Z_{\ge0}^s$, and a partition $\{1,\dots,n\} = A_1 \cup \cdots \cup A_l$ with $A_k \ne \emptyset$ 
 for $k = 1,\dots,l$.  We denote the partition by $\mathcal{A} = \{A_k\}_{k=1}^l$.  This gives monomial ideals
\begin{equation}
\label{monideals}
I_k = \langle \mathbf{x}^{\mathbf{v}_i} \mid i \in A_k\rangle \subseteq R,\quad k = 1,\dots,l.
\end{equation}
For auxiliary variables $\mathbf{t}=(t_1,\dots,t_l)$, the \emph{multi-Rees algebra of $I_1,\dots,I_l$} is the multigraded $R$-algebra
\[
\mathcal{R}_R(I_1{\oplus} \cdots {\oplus}I_l) = \bigoplus_{a_1,\dots,a_l \ge 0} I_1^{a_1} \cdots I_l^{a_l} t_1^{a_1}\cdots t_l^{a_l} \subseteq R[\mathbf{t}].
\]

We construct a presentation of $\mathcal{R}_R(I_1{\oplus} \cdots {\oplus}I_l)$ in the usual way.  For variables $\mathbf{K} = (K_1,\dots,K_n)$, consider the $R$-algebra homomorphism
\begin{equation}
\label{phidef}
\varphi : \K[\mathbf{x},\mathbf{K}] = R[\mathbf{K}] \longrightarrow R[\mathbf{t}]
\end{equation}
that is the identity on $R$ and maps $K_i$ to $\mathbf{x}^{\mathbf{v}_i} t_k$ when $i \in A_k$.  It is easy to see that the image of $\varphi$ is the multi-Rees algebra, so that $\varphi$ induces an $R$-algebra isomorphism
\[
\mathcal{R}_R(I_1{\oplus} \cdots {\oplus}I_l) \simeq \K[\mathbf{x},\mathbf{K}] / \ker(\varphi).
\]
We call $\ker(\varphi) \subseteq \K[\mathbf{x},\mathbf{K}]$ the \emph{defining ideal of the multi-Rees algebra}.

\section{The Main Results}
\label{results}

Our main object of study is the ideal
\[
T_{\mathcal{A}} = \langle K_i \mathbf{x}^{\mathbf{v}_j} - K_j \mathbf{x}^{\mathbf{v}_i} \mid i,j \in A_k \text{ for some } k\rangle : (x_1\cdots x_s)^\infty \subseteq \K[\mathbf{x},\mathbf{K}].
\]

\begin{theorem}
\label{TisRees}
$T_{\mathcal{A}} \subseteq \K[\mathbf{x},\mathbf{K}]$ is the defining ideal of $\mathcal{R}_R(I_1{\oplus} \cdots {\oplus}I_l)$.
\end{theorem}

\begin{proof}
We prove $T_{\mathcal{A}}= \ker(\varphi)$ as follows.  ($\subseteq$) Since $\varphi(K_i) = \mathbf{x}^{\mathbf{v}_i} t_k$ when $i \in A_k$, the inclusion
\begin{equation}
\label{Tprimedef}
T'_{\mathcal{A}} :=  \langle K_i \mathbf{x}^{\mathbf{v}_j} - K_j \mathbf{x}^{\mathbf{v}_i} \mid i,j \in A_k \text{ for some } k\rangle \subseteq \ker(\varphi)
\end{equation}
is immediate.  Now take $f \in T_{\mathcal{A}}$.  Then $(x_1\cdots x_s)^N f \in T'_{\mathcal{A}}$ for some $N$, so that by the above inclusion,
\[
0 = \varphi((x_1\cdots x_s)^N f) = (x_1\cdots x_s)^N \varphi(f)
\]
in $R[\mathbf{t}] = \K[\mathbf{x},\mathbf{t}]$.  This implies $\varphi(f) = 0$, so $f \in \ker(\varphi)$.

($\supseteq$) We begin with some simple algebra.  For each $k =1,\dots,l$, pick $i_k \in A_k$ and observe that 
\begin{equation}
\label{Ksubstitution}
K_i = \frac{1}{\mathbf{x}^{\mathbf{v}_{i_k}}}\big(
K_i \mathbf{x}^{\mathbf{v}_{i_k}} - K_{i_k} \mathbf{x}^{\mathbf{v}_i}\big) + K_{i_k} \mathbf{x}^{\mathbf{v}_i - \mathbf{v}_{i_k}}. 
\end{equation}
Now suppose we have $f \in \K[\mathbf{x},\mathbf{K}]$.  Working in the partial Laurent polynomial ring $\K[\mathbf{x}^{\pm1},\mathbf{K}]$, we can use the substitution \eqref{Ksubstitution} for all $i \in A_k\setminus \{i_k\}$ to obtain an expression of the form
\[
f = \sum_{i=1}^l \sum_{i \in A_k\setminus\{i_k\}} b_i (K_i \mathbf{x}^{\mathbf{v}_{i_k}} - K_{i_k} \mathbf{x}^{\mathbf{v}_i}) + r,
\]
where $b_i \in \K[\mathbf{x}^{\pm1},\mathbf{K}]$ and $r \in \K[\mathbf{x}^{\pm1},K_{i_1},\dots,K_{i_l}]$.  Multiplying by a suitable power of $x_1\cdots x_s$ to clear denominators gives
\begin{equation}
\label{fequation1}
(x_1 \cdots x_s)^N f = \sum_{i \in A_k\setminus\{i_k\}} c_i (K_i \mathbf{x}^{\mathbf{v}_{i_k}} - K_{i_k} \mathbf{x}^{\mathbf{v}_i}) + s,
\end{equation}
where $c_i \in \K[\mathbf{x},\mathbf{K}]$ and $s \in \K[\mathbf{x},K_{i_1},\dots,K_{i_l}]$.

Now suppose that $f \in \ker(\varphi)$.  Applying $\varphi$ to \eqref{fequation1} implies that $\varphi(s) = 0$.  However,  the map $\varphi$ is injective on $\K[\mathbf{x},K_{i_1},\dots,K_{i_l}]$.  To see why, suppose that $g \in \K[\mathbf{x},K_{i_1},\dots,K_{i_l}]$ satisfies $\varphi(g) = 0$.  If we write $g$ as
\[
g = \sum_{a_1,\dots,a_l \ge 0} g_{a_1,\dots,a_l}(\mathbf{x}) K_{i_1}^{a_1}\cdots K_{i_l}^{a_l},
\]
then $\varphi(K_{i_k}) = \mathbf{x}^{\mathbf{v}_{i_k}} t_k$ gives
\begin{align*}
0 = \varphi(g) &=  \sum_{a_1,\dots,a_l \ge 0} g_{a_1,\dots,a_l}(\mathbf{x}) (\mathbf{x}^{\mathbf{v}_{i_1}} t_{1})^{a_1}\cdots  (\mathbf{x}^{\mathbf{v}_{i_l}}t_{l})^{a_l}\\
&=  \sum_{a_1,\dots,a_l \ge 0} g_{a_1,\dots,a_l}(\mathbf{x}) \,\mathbf{x}^{\sum_j a_j\mathbf{v}_{i_j}} \, t_{1}^{a_1}\cdots  t_{l}^{a_l},
\end{align*}
which implies that $g_{a_1,\dots,a_l}(\mathbf{x}) \,\mathbf{x}^{\sum_j a_j\mathbf{v}_{i_j}} = 0$ for all $a_1,\dots,a_l \ge 0$.  It follows that $g = 0$, proving injectivity.  In particular, $\varphi(s) = 0$ implies $s = 0$, so that \eqref{fequation1} becomes
\[
(x_1 \cdots x_s)^N f = \sum_{i \in A_k\setminus\{i_k\}} c_i (K_i \mathbf{x}^{\mathbf{v}_{i_k}} - K_{i_k} \mathbf{x}^{\mathbf{v}_i}) \in T'_{\mathcal{A}},
\]
from which we conclude that $f \in T'_{\mathcal{A}} : (x_1 \cdots x_s)^\infty = T_{\mathcal{A}}$.
\end{proof}

\begin{corollary}
\label{TisToric}
$T_{\mathcal{A}}$ is a toric ideal.
\end{corollary}

\begin{proof} 
We need to prove that $T_{\mathcal{A}}$ is prime and generated by binomials.  Primality follows from Theorem~\ref{TisRees} since $T_{\mathcal{A}}= \ker{\varphi}$ is the kernel of a map to an integral domain.  It remains to find binomial generators for $T_{\mathcal{A}}$.  To do this, introduce a new variable $y$ and set 
\[
T''_{\mathcal{A}} = \langle yx_1\cdots x_s-1,K_i \mathbf{x}^{\mathbf{v}_j} - K_j \mathbf{x}^{\mathbf{v}_i} \mid i,j \in A_k \text{ for some } k\rangle \subseteq \K[\mathbf{x},\mathbf{K},y].
\]
This ideal has binomial generators.  Applying the division algorithm and the Buchberger algorithm to these generators, we see that any reduced Gr\"obner basis $\mathcal{G}$ of $T''_{\mathcal{A}}$ consists of binomials.  However, it is well known that
\[
T_{\mathcal{A}}= T'_{\mathcal{A}} : (x_1\cdots x_s)^\infty = T''_{\mathcal{A}} \cap \K[\mathbf{x},\mathbf{K}].
\]
For any monomial order on $\K[\mathbf{x},\mathbf{K},y]$ that eliminates $y$, the Elimination Theorem tells us that $\mathcal{G} \cap \K[\mathbf{x},\mathbf{K}]$ is the desired binomial generating set of $T$.
\end{proof}

\section{Application to Toric Dynamical Systems}
\label{application}

The paper \emph{Toric dynamical systems} \cite{TDS} attaches various toric ideals to a chemical reaction network.  Such a network is defined by a directed graph $G = (V,E)$ with vertex set $V = \{1,\dots,n\}$.  We assume that $G$ is \textit{weakly reversible}, which means that every connected component of the underlying graph is strongly connected as a directed graph, i.e., given any two vertices $i, j \in V$ there is a directed path from $i$ to $j$ and a directed path from $j$ to $i$.  Then:
\begin{itemize}
\item Each directed edge $i{\to}j$ of $G$ represents a chemical or biochemical reaction with reaction rate $\kappa_{ij}$.
\item Each vertex $i$ of $G$ supports an exponent vector $\mathbf{v}_i \in \Z_{\ge0}^s$ that explains how the vertex is built from molecules or cells called \emph{species} in the literature.  It is customary to assemble the $\mathbf{v}_i$ into an $s \times n$ matrix $Y$ with columns $\mathbf{v}_1,\dots,\mathbf{v}_n$.
\item The concentrations of the species are represented by the variables $x_1,\dots,x_s$.
\end{itemize}
Here is a classic example due to Edelstein \cite{edelstein} which has been studied by many authors, including \cite{feinberglectures,gatermann4}. 

\begin{example}
\label{Ex1}
Consider the reaction network
\begin{equation}
\label{Ex1network}
\begin{matrix} \ce{A <=>[\kappa_{12}][\kappa_{21}] 2A} \\ \ce{A + B <=>[\kappa_{34}][\kappa_{43}] C <=>[\kappa_{45}][\kappa_{54}] B}\end{matrix} 
\qquad \ Y = \begin{pmatrix}1 & \!2 & \!1 & \!0 & \!0\\ 0 & \!0 & \!1 & \!0 & \!1\\ 0 & \!0 & \!0 & \!1 & \!0\end{pmatrix}.
\end{equation}
Here, we have three species $\ce{A},\ce{B},\ce{C}$, and the columns of $Y$ show which combinations of the species appear at each vertex of the directed graph $G$.  The variables $x_1,x_2,x_3$ give the respective concentrations of $\ce{A},\ce{B},\ce{C}$, and the corresponding monomials are
\begin{equation}
\label{Ex1monomials}
\mathbf{x}^{\mathbf{v}_1} = x_1,\ \mathbf{x}^{\mathbf{v}_2} = x_1^2, \ \mathbf{x}^{\mathbf{v}_3} = x_1 x_2, \ \mathbf{x}^{\mathbf{v}_4} = x_3,\ \mathbf{x}^{\mathbf{v}_5} = x_2.
\end{equation}
The graph $G$ has two connected components, each of which is strongly connected.  In the notation of Section~\ref{notation}, we have the partition $V = \{1,2\} \cup \{3,4,5\} = A_1 \cup A_2$, and the resulting monomial ideals are
\[
I_1 = \langle x_1,x_1^2\rangle = \langle x_1\rangle, \ I_2 = \langle x_1x_2, x_3,x_2\rangle = \langle x_2,x_3\rangle.
\]
The monomials in \eqref{Ex1monomials} are non-minimal generators of $I_1, I_2$ and hence give a rather inefficient presentation of the multi-Rees algebra $\mathcal{R}_R[I_1{\oplus} I_2]$ for $R = \Q[x_1,x_2,x_3]$.  But for the purposes of understanding the chemistry of the network \eqref{Ex1network}, the presentation coming from \eqref{Ex1monomials} is the one we want.
\end{example}

The notation in our paper is similar to the notation of \cite{TDS}, except that the variables $c_i$  and exponent vectors $y_i$ of \cite{TDS} are denoted $x_i$ and  $\mathbf{v}_i$ respectively.  Also, we use an arbitrary field $\K$ while  \cite{TDS} works primarily over $\Q$ when doing algebra.

\subsection{The Toric Ideal}  The paper \cite{TDS} first defines the ideal $T_G \subseteq \K[\mathbf{x},\mathbf{K}]$ in the special case when $G$ is strongly connected:
\[
T_G = \langle K_i \mathbf{x}^{\mathbf{v}_j} - K_j \mathbf{x}^{\mathbf{v}_i}\rangle : (x_1\cdots x_s)^\infty \subseteq \K[\mathbf{x},\mathbf{K}].
\]
In general, $G$ will have connected components $G_1,\dots,G_l$, each of which is strongly connected by our hypothesis of weak reversibility.  Then $T_G$ is defined to be
\begin{equation}
\label{TGdef}
T_G = (T_{G_1} + \cdots + T_{G_l}) : (x_1\cdots x_s)^\infty \subseteq \K[\mathbf{x},\mathbf{K}].
\end{equation}
The ideals $T_G$ relate nicely to the ideals $T_\mathcal{A}$ defined in Section~\ref{results} as follows.

\begin{proposition} \
\label{TGisTApartition} 
\begin{enumerate}
\item Given any reaction network with directed graph $G$, we have $T_G = T_\mathcal{A}$, where $\mathcal{A}$ is the partition of the vertex set $V = \{1,\dots,n\}$ induced by the connected components of $G$.
\item Given any partition $\mathcal{A} = \{A_k\}_{k=1}^l$ of $\{1,\dots,n\}$, we have $T_\mathcal{A} = T_G$, where $G$ is the directed graph whose connected component $G_k$ is the complete directed graph with vertex set $A_k$.
\end{enumerate}
\end{proposition}

\begin{proof}
Assertion (1) is obvious when $G$ is strongly connected.  Now suppose that $G$ has two strongly connected components.  This induces a partition $\mathcal{A} = \{A_1,A_2\}$ of the vertices $\{1,\dots,n\}$ of $G$.  For $k=1,2$, define
\[
B_k = \{i \in \{1,\dots,s\} \mid x_i \text{ corresponds to a species appearing in } G_k\}.
\]
Then
\begin{align*}
T_G &= \big(T_{G_1} + T_{G_2}\big):(x_{1}\cdots x_{s})^{\infty}\\ &=
\big(T'_{A_{1}}\hskip-2pt:(\textstyle{\prod_{i\in B_1} x_i})^{\infty}+T'_{A_{2}}\hskip-2pt:(\textstyle{\prod_{i\in B_2} x_i})^{\infty}\big):(x_{1}\cdots x_{s})^{\infty}\\
T_\mathcal{A} &= (T'_{A_{1}}+T'_{A_{2}}):(x_{1}\cdots x_{s})^{\infty},
\end{align*}
where $T'_{A_{k}}=\langle K_{i}\mathbf{x}^{\mathbf{v}_{j}}-K_{j}\mathbf{x}^{\mathbf{v}_i} \mid i,j\in A_{k}\rangle$ for $k=1,2$. It is straightforward to prove that these ideals are equal.  The general case is similar, and then we are done since (2) follows from (1).
\end{proof}

We then have the following result.

\begin{theorem}\ 
\label{TTG}
\begin{enumerate}
\item $T_G$ is the defining ideal of the multi-Rees algebra  $\mathcal{R}_R(I_1{\oplus} \cdots {\oplus}I_l)$ for the monomial ideals of the partition given by the connected components of $G$.
\item $T_G$ is a toric ideal.
\end{enumerate}
\end{theorem}

\begin{proof}
(1) is an immediate consequence of $T_G = T_\mathcal{A}$ and Theorem~\ref{TisRees}, and then (2) follows from Corollary~\ref{TisToric}. 
\end{proof}

\begin{remark} In \cite{TDS}, part (2) of Theorem~\ref{TTG} is proved  in \cite[Prop.\ 6]{TDS} by a different argument, assuming that $G$ is strongly connected.  In the general case, when $T_G$ is defined by \eqref{TGdef}, \cite{TDS} implicitly assumes without proof that $T_G$ is toric.
\end{remark}

\begin{remark} The definition of chemical reaction network given at the start of the section allows the graph $G$ to have isolated vertices.  However, from the chemical perspective, the reactions are the heart of the matter.  Since these are the directed edges of $G$, in practice one typically assumes that every connected component of $G$ has at least one edge.  In our notation, this corresponds to $|A_k| > 1$ for all $k$.
\end{remark}

\subsection{Adding Edges} One observation is that the ideal $T_G$ has only a modest dependence on the edges of $G$, since $T_G$ is completely determined by the partition of the vertex set  $V = A_1 \cup  \cdots \cup A_l$.  For instance, suppose we add a new reaction $\ce{A + B ->[\kappa_{35}] B}$ to the network in Example~\ref{Ex1network}. This would give a larger directed graph $G'$ (one more edge), yet we have $T_{G'} = T_G$ since $G'$ and $G$ give the same partition of $V$.  

We record this observation more formally as follows:

\begin{proposition}
\label{AddEdgeTG}
In a chemical reaction network, adding a directed edge to $G$ that connects vertices within the same connected component has no effect on the toric ideal $T_G$. 
\end{proposition}

The ability to add edges is exploited in the proof of \cite[Prop.\ 6]{TDS}, where the authors replace $G$ (strongly connected in Prop.\ 6) with the complete directed graph on the vertices of $G$. We also use this fact the second part of Proposition \ref{TGisTApartition}.

\subsection{The Toric Moduli Space} Another important ideal defined in \cite{TDS} is the \emph{moduli ideal}
\[
M_G = T_G \cap \K[\mathbf{K}],
\]
whose variety parametrizes the toric dynamical systems which have a positive solution that is
balanced at the level of complexes.  More precisely, \cite[Thm.\ 7]{TDS} implies that the positive real points of this variety parametrize choices of reaction rates $\kappa_{ij}$ that give rise to a toric dynamical system. (The full story uses the formula for $K_i$ in terms of the $\kappa_{ij}$ coming from the Matrix Tree Theorem.  This is explained in \cite[Section 2]{TDS}.)  As noted in \cite{TDS}, $M_G$ is a toric ideal (which also follows easily from our results).

Proposition~\ref{AddEdgeTG} implies the following result about $M_G$.

\begin{corollary}
\label{AddEdgeMG}
In a chemical reaction network, adding a directed edge to $G$ that connects vertices within the same connected component has no effect on the moduli ideal $M_G$. 
\end{corollary}

\subsection{The Stoichiometric Subspace and the Deficiency}  Another important player in this theory is the \emph{stoichiometric subspace}
\[
S = \mathrm{Span}_\R(\mathbf{v}_j-\mathbf{v}_i \mid i{\to}j \in E) \subseteq \R^s,
\]
and then the \emph{deficiency} of the network is defined to be
\begin{equation}
\label{deltadef}
\delta = n - l - \dim_{\R} S,
\end{equation}
where $n = \#$vertices of $G$ and $l = \#$connected components of $G$.  The observation is that neither of these is affected by adding an edge to $G$ that connects vertices within the same connected component (which is strongly connected by weak reversibility).  To see why $S$ is unchanged, let $i, j$ be vertices in the same connected component, and suppose $G$ has no directed edge from $i$ to $j$.  By strong connectivity, we have a sequence of directed edges
\[
i \to i' \to i'' \to \cdots \to i^{(\ell)} \to j.
\]
In the stoichiometric subspace $S$, these edges give the telescoping sum
\[
(\mathbf{v}_{i'} - \mathbf{v}_i) + (\mathbf{v}_{i''} -\mathbf{v}_{i'}) + \cdots + (\mathbf{v}_j -\mathbf{v}_{i^{(\ell)}}) = \mathbf{v}_j - \mathbf{v}_i,
\]
Thus $\mathbf{v}_j - \mathbf{v}_i \in S$, so adding a new directed edge $i{\to}j$ to $G$ has no effect on $S$ and hence on $\delta$ (since we change neither $l$ nor $S$).  Hence we have proved:

\begin{proposition}
\label{AddEdgeSdelta}
In a chemical reaction network, adding a directed edge to $G$ that connects vertices within the same connected component has no effect on the stoichiometric subspace $S$ and the deficiency $\delta$. 
\end{proposition}

There is a nice relation between Corollary~\ref{AddEdgeMG} and Proposition~\ref{AddEdgeSdelta}.  By \cite[Thm.\ 9]{TDS}, the codimension of the moduli ideal $M_G$ equals the deficiency $\delta = n - l - \dim_{\R} S$.  So when we add an edge within a connected component, $M_G$ doesn't change by Corollary~\ref{AddEdgeMG}, which means that the codimension, hence the deficiency, doesn't change. Proposition~\ref{AddEdgeSdelta} gives the intrinsic reason why the deficiency is unchanged.

\section{The Cayley Matrix} 
\label{cayley}

A key tool used in the proof of \cite[Thm.\ 9]{TDS} is the Cayley matrix described as follows.  If we group the columns $\mathbf{v}_i$ of $Y$ according to which connected component of $G = G_1 \cup \cdots \cup G_l$ they lie in, and renumber appropriately, we can write $Y$ in the form
\[
Y = \big( Y_1 \mid Y_2 \mid \cdots \mid Y_l\big)
\]
where the column indices give the partition $\{1,\dots,n\} = A_1  \cup \cdots \cup A_l$ used to define the ideals $I_k$ from \eqref{monideals}.  Following \cite{TDS}, we add $l$ rows to the bottom of $Y$ to obtain the \emph{Cayley matrix}
\[
\mathrm{Cay}_G(Y) = \begin{pmatrix} Y_1 & Y_2 & \cdots & Y_l\\
\mathbf{1} & \mathbf{0} & \cdots & \mathbf{0}\\ 
\mathbf{0} & \mathbf{1} & \cdots & \mathbf{0}\\
\vdots & \vdots & \ddots & \vdots\\
\mathbf{0} & \mathbf{0} & \cdots & \mathbf{1}
\end{pmatrix},
\]
where $\mathbf{0}$ and $\mathbf{1}$ are row vectors of all 0's and all 1's respectively of suitable length.

The Cayley matrix is used in the proof of \cite[Thm.\ 9]{TDS}, which relates the codimension of $M_G$ to the deficiency $\delta$ from \eqref{deltadef}.  In \cite[Rem.\ 8]{TDS}, the authors observe a direct connection to deficiency:
\[
\delta = \mathrm{nullity}(\mathrm{Cay}_G(Y))
\]
In comments following the proof of  \cite[Thm.\ 9]{TDS}, the authors note that the moment map gives a bijection between the positive part of the variety of $M_G$ (which parametrizes toric dynamical systems) and the interior of the corresponding \emph{Cayley polytope}, which is the convex hull of the columns of $\mathrm{Cay}_G(Y)$.  Cayley polytopes have been studied extensively and have many applications. See, for example, \cite{DDRP,ito}.  

From our point of view, each row of $\mathrm{Cay}_G(Y)$ is associated to a variable:\ the first $s$ rows correspond to $x_1,\dots,x_s$, and the last $l$ rows correspond to the auxiliary variables $t_1,\dots,t_l$ used in the construction of the multi-Rees algebra.  Using these variables, the $i$th column of the Cayley matrix gives the monomial $\mathbf{x}^{\mathbf{v}_i} t_k$ when $i \in A_k$.  Since these monomials generate the multi-Rees algebra over $R$, we have proved the following result:

\begin{proposition}
\label{CayleyRees}
The monomials coming from the columns of the Cayley matrix $\mathrm{Cay}_G(Y)$ are $R$-algebra generators of the multi-Rees algebra $\mathcal{R}_R(I_1{\oplus} \cdots {\oplus}I_l)$.  
\end{proposition}

This proposition means that from the algebraic viewpoint, the Cayley matrix leads to the multi-Rees algebra, while from the geometric viewpoint, the same matrix leads to  the Cayley polytope.

But there is more to say, since the columns of any integer matrix, when regarded as exponent vectors, give a toric ideal in the usual way.  Here is how this works in our situation.

\begin{proposition}\
\label{cayleytoric}
\begin{enumerate}
\item $M_G$ is the toric ideal of the Cayley matrix $\mathrm{Cay}_G(Y)$.
\item $T_G$ the toric ideal of the modified Cayley matrix
\[
\begin{pmatrix} I_s &Y_1 & Y_2 & \cdots & Y_l\\
\mathbf{0} &\mathbf{1} & \mathbf{0} & \cdots & \mathbf{0}\\ 
\mathbf{0} &\mathbf{0} & \mathbf{1} & \cdots & \mathbf{0}\\
\vdots & \vdots & \vdots & \ddots & \vdots\\
\mathbf{0} &\mathbf{0} & \mathbf{0} & \cdots & \mathbf{1}
\end{pmatrix},
\] 
where $I_s$ is the $s\times s$ identity matrix. 
\end{enumerate}
\end{proposition}

\begin{proof}
For (1), we noted above that the columns of $\mathrm{Cay}_G(Y)$ are $\mathbf{x}^{\mathbf{v}_i} t_k$ for $i \in A_k$, giving  the map
\begin{equation}
\label{phirestricted}
\K[\mathbf{K}] \longrightarrow \K[\mathbf{x},\mathbf{t}]
\end{equation}
that sends $K_i$ to $\mathbf{x}^{\mathbf{v}_i} t_k$ for $i \in A_k$.  This is the restriction to $\K[\mathbf{K}]$ of the map $\varphi$ defined in \eqref{phidef}.  By definition, the kernel of \eqref{phirestricted} is the toric ideal of  $\mathrm{Cay}_G(Y)$.  Using Theorem~\ref{TisRees} and the definition of $M_G$, we obtain
\[
\text{kernel of \eqref{phirestricted}} = \ker(\varphi)\cap \K[\mathbf{K}] = T_G \cap \K[\mathbf{K}] = M_G,
\]
which completes the proof of (1).  

For (2), the modified Cayley matrix gives the monomials $x_1,\dots,x_s$ (from the first $s$ columns) and $\mathbf{x}^{\mathbf{v}_i} t_k$ for $i \in A_k$ (from the remaining $n$ columns).  These monomials give  the map $\varphi$ from \eqref{phidef}, so that the toric ideal of the modified matrix is $\ker(\varphi)$, which by Theorem~\ref{TisRees} is precisely $T_G$. 
\end{proof}

\begin{remark}The proof of \cite[Thm.\ 9]{TDS} uses the \emph{extended Cayley matrix}
\[
\begin{pmatrix} -I_s &Y_1 & Y_2 & \cdots & Y_l\\
\mathbf{0} &\mathbf{1} & \mathbf{0} & \cdots & \mathbf{0}\\ 
\mathbf{0} &\mathbf{0} & \mathbf{1} & \cdots & \mathbf{0}\\
\vdots & \vdots & \vdots & \ddots & \vdots\\
\mathbf{0} &\mathbf{0} & \mathbf{0} & \cdots & \mathbf{1}
\end{pmatrix},
\]
and the authors comment that ``The toric ideal of this matrix is precisely the toric balancing ideal $T_G$.''  This puzzles us, for the minus sign in front of $I_s$ means that $x_i$ would map to $x_i^{-1}$.  Simple examples show that this does not give the same toric ideal as the modified Cayley matrix used in Proposition~\ref{cayleytoric}.  For instance, in the situation of Example~\ref{Ex1}, $T_G$ contains $K_1x_1-K_2$ (we will confirm this in Example~\ref{Ex1revisit} below), while the extended Cayley matrix gives a toric ideal that contains $K_1 - K_2x_1$ since $x_1 \mapsto x_1^{-1}$ implies $K_1 - K_2x_1 \mapsto (x_1t_1) - (x_1^2t_1)x_1^{-1} = 0$.
\end{remark}

\section{The Special Fiber}
\label{special}

In Section~\ref{application}, we saw that the moduli ideal $M_G = T_G \cap \K[\mathbf{K}]$ plays an important role in the theory of toric dynamical systems.  But for the ideals $T_\mathcal{A} \subseteq \K[\mathbf{x},\mathbf{K}]$ from Section~\ref{results}, one can ask if $T_\mathcal{A} \cap \K[\mathbf{K}]$ has a meaning from the Rees algebra point of view.  This is a reasonable question since $T_\mathcal{A}$ is the defining ideal of $\mathcal{R}_R(I_1{\oplus} \cdots {\oplus}I_l)$.  However, as we will soon see, the answer is \emph{sometimes yes, sometimes no}.  

To understand why, we need to recall the \emph{special fiber} of a Rees algebra.  For the multi-Rees algebra $\mathcal{R}_R(I_1{\oplus} \cdots {\oplus}I_l)$, the special fiber is the $\K$-algebra
\[
\mathcal{F}_\K(I_1{\oplus} \cdots {\oplus}I_l) = \mathcal{R}_R(I_1{\oplus} \cdots {\oplus}I_l) \otimes_R \K =  \mathcal{R}_R(I_1{\oplus} \cdots {\oplus}I_l)/\langle \mathbf{x}\rangle \mathcal{R}_R(I_1{\oplus} \cdots {\oplus}I_l),
\]
where we regard $\K$ as an $R$-algebra via $\K \simeq R/\langle \mathbf{x}\rangle$.  Here is one case where it is possible to compute the special fiber directly from the defining ideal $T_\mathcal{A}$:

\begin{proposition}
\label{degreegenscase}
Assume that $x_1,\dots,x_s$ have positive weights $q_1,\dots,q_s$ such that for all $k = 1,\dots,l$, the monomials $\mathbf{x}^{\mathbf{v}_i}$ for $i \in A_k$ all the same weighted degree $d_k$.  Then:
\begin{enumerate} 
\item $T_\mathcal{A} \subseteq \mathbb{K}[\mathbf{x},\mathbf{K}]$ is homogeneous with respect to $(\Z_{\ge0} \times \Z_{\ge0}^l)$-multigrading on $\mathbb{K}[\mathbf{x},\mathbf{K}]$ defined by
\[
\mathrm{mdeg}(x_j) = (q_j,\mathbf{0}),\ \mathrm{mdeg}(K_i) = (0,\mathbf{e}_k),\ i \in A_k,
\]
where where $\mathbf{e}_1,\dots,\mathbf{e}_l$ are the standard basis of $\Z_{\ge0}^l$.  
\item There is a natural $\K$-algebra isomorphism
\[
\K[\mathbf{K}]/(T_{\mathcal{A}} \cap \K[\mathbf{K}]) \simeq \mathcal{F}_{\mathbb{K}}(I_{1}{\oplus}\cdots{\oplus} I_{l}).
\]
\end{enumerate}
\end{proposition}

\begin{proof}
For (1), observe that the hypothesis of the proposition implies that 
\[
\mathrm{mdeg}( K_i \mathbf{x}^{\mathbf{v}_j}) = (d_k,\mathbf{e}_k)
\]
whenever $i,j \in A_k$.  It follows that the ideal 
\[
T'_{\mathcal{A}} :=  \langle K_i \mathbf{x}^{\mathbf{v}_j} - K_j \mathbf{x}^{\mathbf{v}_i} \mid i,j \in A_k \text{ for some } k\rangle 
\]
from \eqref{Tprimedef} is homogeneous with respect to this mutligrading.  Since $x_1\cdots x_s$ is also homogeneous, the same is true for 
\[
T'_{\mathcal{A}} : (x_1\cdots x_s)^\infty = T_{\mathcal{A}},
\]
where the final equality is from the proof of Theorem~\ref{TisRees}.  Hence we have proved that $T_{\mathcal{A}}$ is homogeneous with respect to the multigrading  on $\mathbb{K}[\mathbf{x},\mathbf{K}]$ given by $\mathrm{mdeg}$, proving (1).  

For (2), we use the homomorphism $\varphi$ defined in \eqref{phidef} and Theorem~\ref{TisRees} to obtain the short exact sequence
\[
0\longrightarrow T_{\mathcal{A}}\longrightarrow R[\mathbf{K}]\longrightarrow\mathcal{R}_{R}(I_{1}{\oplus}\cdots{\oplus} I_{l})\longrightarrow0.
\]
Then we tensor this sequence over $R$ with $R$-module $\mathbb{K}$ to
obtain the exact sequence
\begin{equation}
\label{TAtensorK}
 T_{\mathcal{A}}\otimes_{R}\mathbb{K}\longrightarrow\mathbb{K}[\mathbf{K}]\longrightarrow\mathcal{F}_{\mathbb{K}}(I_{1}{\oplus}\cdots{\oplus} I_{l})\longrightarrow0,
\end{equation}
and the isomorphism
\[
\mathcal{F}_{\mathbb{K}}(I_{1}{\oplus}\cdots{\oplus} I_{l})\simeq\mathbb{K}[\mathbf{K}]/\mathrm{im}(T_{\mathcal{A}}\otimes_{R}\mathbb{K}).
\]
We need to prove that $\mathrm{im}(T_{\mathcal{A}}\otimes_{R}\mathbb{K}) = T_{\mathcal{A}} \cap \K[\mathbf{K}]$.

We know $T_{\mathcal{A}} = \ker(\varphi)$ is generated by binomials of the form $\mathbf{x}^{\alpha_1} \mathbf{K}^{\beta_1} - \mathbf{x}^{\alpha_2} \mathbf{K}^{\beta_2}$.  By (1), these binomials must have the same multidegree (otherwise, each monomial would be in $\ker(\varphi)$, clearly impossible).  Hence $\mathbf{x}^{\alpha_1}$ and $\mathbf{x}^{\alpha_2}$ have the same weighted degree, so that in particular,  $\mathbf{x}^{\alpha_1} = 1$ if and only if $\mathbf{x}^{\alpha_2} = 1$.  Since $T_\mathcal{A} \to T_\mathcal{A} \otimes \K$ sends all $x_i$ to $0$, we see that the tensor product is generated by minimal generators of the form $\mathbf{K}^{\beta_1} - \mathbf{K}^{\beta_2}$, which are precisely the minimal generators of $T_\mathcal{A}\cap \K[\mathbf{K}]$.  This proves that image of the left-most map in \eqref{TAtensorK} is precisely $T_\mathcal{A}\cap \K[\mathbf{K}]$.
\end{proof}

\begin{remark} 
\label{BrunsConca}
The isomorphism $\mathcal{F}_{\mathbb{K}}(I_{1}{\oplus}\cdots{\oplus} I_{l}) \simeq \K[\mathbf{K}]/(T_{\mathcal{A}} \cap \K[\mathbf{K}])$ from Proposition~\ref{degreegenscase} is equivalent to saying that the special fiber can be identified with the $\K$-subalgebra of the multi-Rees algebra generated by $\mathbf{x}^{\mathbf{v}_i}t_k$ for $i \in A_k$.  In the standard graded case when $q_1 = \cdots = q_s =1$, this description of the special fiber is due to Bruns and Conca \cite[Remark 3.2(c)]{BC}.  We also note that the multigrading introduced in \cite[Section 3]{BC} is equivalent to the one defined in Proposition~\ref{degreegenscase} via an automorphism of $\Z_{\ge0} \times \Z_{\ge0}^l$.
\end{remark}

We now compute some examples of special fibers.

\begin{example}
\label{degreegensEx}
Here is an example from \cite{TDS}.  Consider the reaction network
\smallskip
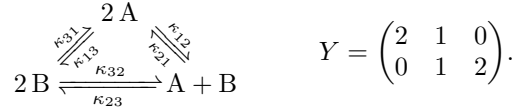
\begin{center}
\begin{tabular}{cc}
{\psset{unit=.6pt}\begin{pspicture}(159,72)
\rput[bl](55,58){\ce{2A}}
\rput[bl](0,0){\ce{2B   <=>[\hskip10pt\kappa_{32}\hskip10pt][\kappa_{23}] A + B}}
\rput[bl]{45}(36,16){\small\ce{<=>[\kappa_{31}][\kappa_{13}]}}
\rput[bl]{-45}(73,46){\small\ce{<=>[\kappa_{12}][\kappa_{21}]}}
\end{pspicture}} & \quad \raisebox{20pt}{$Y = \begin{pmatrix} 2 & 1 & 0\\ 0 & 1 & 2\end{pmatrix}$.}
\end{tabular}
\end{center}

\noindent We have variables $x_1,x_2$ and the monomial ideal $I = \langle x_1^2, x_1 x_2, x_2^2 \rangle = \langle x_1,x_2\rangle^2$.  One computes that 
\begin{align*}
T_G &= \langle K_1 x_1 x_2 - K_2 x_1^2, K_1 x_2^2 - K_3 x_1^2, K_2 x_2^2 - K_3 x_1x_2\rangle : \langle x_1x_2\rangle^\infty\\
&=\langle K_3x_1 - K_2 x_2, K_2 x_1 - K_1 x_2, K_1 K_3- K_2^2\rangle.
\end{align*}
Then the moduli ideal is $M_G = T_G \cap \K[K_,K_2,K_3] = \langle K_1
K_3- K_2^2\rangle$.  Since $I$ is generated by monomials of
degree two, Proposition~\ref{degreegenscase} implies that the special fiber is
\[
\mathcal{F}_\K(I) \simeq \K[K_,K_2,K_3]/\langle K_1 K_3- K_2^2\rangle.
\]
Geometrically, this says that blowing up the plane $\A^2_\K$ at $I = \langle x_1,x_2\rangle^2$ (the square of the maximal ideal of the origin) has exceptional fiber given by the rational normal curve of degree $2$ in $\PP^2_\K$.  So in this case, the moduli ideal defines the special fiber.  
\end{example}

Here is an example of what can happen when the generators of the
monomial ideals have mixed degree.

\begin{example}
\label{Ex1revisit}
In the situation of Example~\ref{Ex1}, one computes that
\begin{align*}
T_G &= \langle K_1 x_1^2 - K_2 x_1, K_3 x_3 - K_4 x_1 x_2, K_3 x_2 - K_5 x_1 x_2, K_4 x_2 - K_5 x_3\rangle : (x_1x_2x_3)^\infty\\
&= \langle K_1 K_3 - K_2 K_5, K_4 x_2 - K_5 x_3, K_1 x_1-K_2, K_5 x_1-K_3 \rangle\\[2pt]
M_G &= T_G \cap \Q[K_1,\dots,K_5] = \langle K_1 K_3 - K_2 K_5\rangle.
\end{align*}
To understand the special fiber, we use \eqref{TAtensorK}, which here is the exact sequence
\[
T_G \otimes_R \K \longrightarrow \K[K_1,\dots,K_5] \longrightarrow \mathcal{F}_\K(I_1{\oplus}I_2) \longrightarrow 0.
\]
In $T_G \otimes_R \K$, the generators $K_1 K_3 {-} K_2 K_5, K_4 x_2 {-} K_5 x_3, K_1 x_1{-}K_2, K_5 x_1{-}K_3$ of $T_G$ map to  $K_1 K_3 {-} K_2 K_5, 0, {-}K_2, {-}K_3$ in $T_G \otimes_R \K$, so that the special fiber is
\[
\mathcal{F}_\K(I_1{\oplus}I_2) \simeq \K[K_1,\dots,K_5] /\langle K_2,K_3\rangle.
\]
In this case, the moduli ideal $M_G = \langle K_1 K_3 - K_2
K_5\rangle$ gives no information about the special fiber.  This
happens because we use $I_1 = \langle x_1,x_1^2\rangle$ and $I_2 =
\langle x_1x_2, x_3,x_2\rangle$ with mixed degree generating sets, which create the generators $K_1 x_1{-}K_2,K_5 x_1{-}K_3$ of $T_G$. 
\end{example}

\begin{remark}
\label{geometrytoric}
The ($\Z_{\ge0} \times \Z_{\ge0}^l$)-multigrading on $\mathbb{K}[\mathbf{x},\mathbf{K}]$ defined in part (1) of Proposition~\ref{degreegenscase} has a nice geometric interpretation.  For each $k$, the monomials $\mathbf{x}^{\mathbf{v}_i}$, $i \in A_k$, have the same weighted degree $d_k$ and hence define a rational map
\[
\PP(q_1,\dots,q_s) \dashrightarrow \PP^{|A_k|-1}.
\]
Combining these for $k =1,\dots,l$ gives the rational map
\begin{equation}
\label{Fdef}
F : \PP(q_1,\dots,q_s) \dashrightarrow \prod_{k=1}^l \PP^{|A_k|-1}.
\end{equation}
Then $T_\mathcal{A}$ is the defining ideal of the closure of the graph
\begin{equation}
\label{graphclosure}
\overline{\Gamma(F)} \subseteq \PP(q_1,\dots,q_s) \times \prod_{k=1}^l \PP^{|A_k|-1}
\end{equation}
The homogeneous coordinate ring of the product in \eqref{graphclosure} is $\K[\mathbf{x},\mathbf{K}]$ with the multigrading defined in part (1) of Proposition~\ref{degreegenscase}.  From this point of view, it is natural that $T_\mathcal{A}$ is homogeneous with respect to this multigrading.  Furthermore, $\K[\mathbf{x},\mathbf{K}]/T_\mathcal{A} \simeq \mathcal{R}_R(I_1{\oplus}\cdots{\oplus}I_l)$ shows that the  multi-Rees algebra is the homogeneous coordinate ring of the graph closure of the rational map \eqref{Fdef}.  The paper \cite{botbol} studies a special case of the map \eqref{Fdef} (with monomials replaced with polynomials).  

We also mention that this whole construction generalizes, where one can replace $\PP(q_1,\dots,q_s)$ with any complete toric variety $X_\Sigma$.  Here, we assume that the monomials $\mathbf{x}^{\mathbf{v}_i}$, $i \in A_k$, have the same degree in the total coordinate ring of $X_\Sigma$, which is graded by the class group $\mathrm{Cl}(X_\Sigma)$.  Proposition 5.1 continues to hold in this case because $1$ is the unique monomial of degree zero in $\mathrm{Cl}(X_\Sigma)$ since $X_\Sigma$ is complete.
\end{remark}

\section*{Acknowledgements} 

The collaboration that led to this paper began at the CBMS conference \emph{Applications of Polynomial Systems} at Texas Christian University in June 2018.  We are grateful to the Conference Board of the Mathematical Sciences and the National Science Foundation for supporting this conference and to the local organizers for making us feel so welcome. We also thank Bernd Ulrich and the referee for helpful comments. 

\bibliographystyle{amsplain}

\end{document}